\documentclass[12pt,final]{amsart}

\usepackage{graphicx} 
\usepackage[margin=1.1in]{geometry}
\usepackage{amsmath}
\usepackage{physics}

\usepackage{hyperref}
\usepackage{verbatim}
\usepackage{amssymb}
\usepackage{amscd}
\usepackage{xcolor}
\usepackage{biblatex}
\newtheorem{theorem}{Theorem}[section]

\newtheorem{lemma}{Lemma}[section]
\newtheorem{proposition}[lemma]{Proposition}

\newtheorem{corollary}[lemma]{Corollary}
\newtheorem{obs}[lemma]{Observation}

\theoremstyle{definition}

\newtheorem{notation}[lemma]{Notation}

\theoremstyle{remark}

\newtheorem{example}[lemma]{Example}

\newcommand{\N}{\ensuremath{{\mathbb N}}}

\newcommand{\C}{\ensuremath{\mathbb C}}

\newcommand{\HA}{\ensuremath{\mathcal{H}_A}}
\newcommand{\HB}{\ensuremath{\mathcal{H}_B}}
\newcommand{\HAB}{\ensuremath{\mathcal{H}_A\otimes \mathcal{H}_B}}
\newcommand{\QH}{\ensuremath{\mathcal{H}}}

\newcommand{\Pro}{\ensuremath{\mathbb P}}

\newcommand{\lr}{\ensuremath{\mathcal R}}

\author{Manimugdha Saikia}
\address{Department of Mathematics, University of Western Ontario, London Ontario N6A 5B7, Canada}
\email{msaikia@uwo.ca}

\date{\today}

\thanks{}

\title{Restrictions of holomorphic sections to products}
\begin{document}
\sloppy

\maketitle

\noindent {\bf Abstract.} We associate quantum states with subsets of a product of two compact connected K\"ahler manifolds $M_1$ and $M_2$. To associate the quantum state with the subset, we use the map that restricts holomorphic sections of the quantum line bundle over the product of the two K\"ahler manifolds to the subset. We present a description of the kernel of this restriction map when the subset is a finite union of products. This in turn shows that the quantum states associated with the finite union of products are separable. Finally, for every pure state and certain mixed state, we construct subsets of $M_1\times M_2$ such that the states associated with these subsets are the original states, to begin with.

\let\thefootnote\relax
\let\thefootnote\relax

%\tableofcontents

\section{Introduction}
The study of the interplay between geometric structures and analytic objects arising from these structures emerges in many different ways in mathematics. One way of studying this interplay is through geometric quantization. Geometric quantization, due to the work of Kostant, Souriau and Kirillov \cite{Sou67, Kos70}, is a technique that produces Hilbert spaces from symplectic manifolds. For a closed, connected K\"ahler manifold $M$ equipped with a (pre)quantum line bundle $L$, and for any $k \in \N$, we define the corresponding quantum space as the Hilbert space $H^0(M,L^k)$, the space of global holomorphic sections of $L^k\to M$. This provides a framework to investigate various quantum information-theoretic aspects of this Hilbert space and possible relationships with the corresponding manifolds. For instance, in \cite{BPU95}, the authors constructed, for each $k$, a holomorphic section of $H^0(M,L^k)$ associated to a Lagrangian submanifold $\Lambda$ of a compact K\"ahler manifold $M$ and studied various aspects of these states, for example, they computed large $k$ asymptotic of $\langle Tu_k,u_k \rangle$ for a Toeplitz operator $T$ and showed these states concentrate on $\Lambda$. Various similar studies have been made with the ``geometry vs analysis" theme for Lagrangian submanifolds \cite{BW00, DP06, Pao07} and isotropic submanifolds \cite{GU16} of a compact K\"ahler manifold.\\ 

Exploring the correspondence between submanifolds (or, more generally, subsets with nice properties) of K\"ahler manifolds and the corresponding states associated in different ways has appeared in many different contexts. In \cite{BP17}, the authors considered a compact connected complex manifold $M$ and associated a sequence of quantum states with the Lagrangian submanifold $M$ embedded anti-diagonally inside $M\times M$. They showed that this sequence of pure states is a sequence of maximally entangled states. In \cite{barron2022entanglement}, the authors considered a submanifold $\Lambda$ of the product of two integral compact K\"ahler manifolds $M_1$ and $M_2$ having pre-quantum line bundle $L_1$ and $L_2$. For every $N\in \N$, using the map that restricts global holomorphic sections of $L_1^N \boxtimes L_2^N \to M_1\times M_2$, the associated a sequence of mixed states $\rho_N$ with $\Lambda$ and showed that when $\Lambda$ is a product submanifold the states in this sequence are not entangled. Motivated by \cite{barron2022entanglement}, in this article, we come up with a different recipe to associate quantum states with subsets of $M_1\times M_2$ and study the entanglement properties of these newly associated states when we have a product subset. More generally, we showed that when $\Lambda$ is the union of finitely many product subsets, the states are not entangled. 
%\cite{SZ99,LeF18} check why I should input these?

\section{Preliminaries and setup}
Let us first recall a few definitions from Quantum Information Theory. A partial list of good references to know more about quantum information theory can be found here \cite{Nielsen_Chuang_2010, Bengtsson_Zyczkowski_2006}. Let $\QH$ be a finite-dimensional Hilbert space. In the context of quantum information theory, this Hilbert space shall be called a \textit{quantum Hilbert space}. A unit vector $v$ in $\QH$ is called a \textit{pure state} and a hermitian positive semi-definite operator $\sigma$ on $\QH$ with trace 1 is called a \textit{mixed state}. Note that if $v$ is a pure state, then the orthogonal projection operator $\mathcal{P}_v$ onto the one-dimensional vector space spanned by $v$ also has the properties of a mixed state. This way, any pure state can also be considered a more general mixed state. Note that any mixed state $\sigma$, being a hermitian positive semi-definite operator, can be decomposed in the form $\sigma= \sum_{j} p_j \mathcal{P}_{v_j}$ for some pure states $v_j$ and $\sum_j p_j=1$ (for example the eigendecomposition of $\sigma$). However, this decomposition is not unique. We say a mixed state can be written as a convex combination of pure states. Further, by a state, we always mean a mixed state and we explicitly mention a pure state when it is one.\\

When our quantum Hilbert space $\QH$ is a tensor product of two finite-dimensional Hilbert spaces $\HA$ and $\HB$, we can talk about the entanglement of states. A pure state $v \in \HAB$ is called a \textit{separable} or \textit{non-entangled} if we can write $v = u \otimes w$ for some $u \in \HA$ and $w \in \HB$. Otherwise, it is called \textit{entangled}. A mixed state $\sigma$ is called separable or non-entangled if there exists a decomposition of $\sigma = \sum_{j} p_j\mathcal{P}_{v_j}$ such that each of the pure states $v_j$ are separable. Otherwise, the mixed state is called entangled.\\

For $j\in \{ 1,2\}$ let $(L_j,h_j)$ be a holomorphic hermitian line bundle on a compact connected K\"ahler manifold  $(M_j,\omega_j)$ of complex dimension $d_j\ge 1$ such that the curvature of the Chern connection on $L_j$ 
is $-i\omega_j$. A holomorphic hermitian line bundle satisfying the curvature condition is called a \textit{(pre)quantum line bundle}. A line bundle defined above becomes an ample line bundle. On top of that, we will assume that the line bundles are very ample. Let $\mu_j$ be the measure on $M_j$ associated to the volume form $\frac{\omega_j^{d_j}}{d_j!}$ and $\mu$ be the measure on $M_1\times M_2$ associated to the volume form $ \frac{\omega_1^{d_1}\omega_2^{d_2}}{d_1!d_2!}$. Let $\pi_1:M_1\times M_2\to M_1$ and $\pi_2:M_1\times M_2\to M_2$ be the 
projections onto the first and the second manifold respectively. Recall that the holomorphic hermitian line bundle $L_1\boxtimes L_2\to M_1\times M_2$ (the external tensor product of line bundles $L_1 \to M_1$ and $L_2 \to M_2$) is defined by $L_1\boxtimes L_2=\pi_1^*(L_1)\otimes \pi_2^*(L_2)$ 
 and 
 \begin{equation} \label{1}
    H^0(M_1\times M_2,L_1 \boxtimes L_2)\cong
  H^0(M_1,L_1)\otimes H^0(M_2,L_2).
 \end{equation}
The line bundle $L_1 \boxtimes L_2 \to M_1 \times M_2$ satisfies the curvature condition and becomes a quantum line bundle on $M_1 \times M_2$. Further, since the K\"ahler manifolds involved are compact, these spaces are finite-dimensional. The space of global holomorphic sections $H^0(M_j,L_j)$ forms a Hilbert space with the following inner product:
$$\langle s_1, s_2 \rangle_j = \int_{M_j} h_j(s_1(z),s_2(z))d\mu_j(z) \text{ for } s_1,s_2 \in H^0(M_j,L_j).$$
Similarly, the space $H^0(M_1\times M_2,L_1 \boxtimes L_2)$ is a Hilbert space with the inner product $\langle .,. \rangle$ given using $\langle .,.\rangle_1$ and $\langle .,.\rangle_2$ and the relation \eqref{1}. In this article, our quantum Hilbert space is $ H^0(M_1,L_1)\otimes H^0(M_2,L_2)$. For notational convenience, from now onwards we denote $\HA = H^0(M_1,L_1)$ and $\HB = H^0(M_2,L_2)$ so that our quantum Hilbert space is $\HAB$. \\

For $j\in \{1,2\}$, let $\Lambda_j$ be a non-empty subset of $M_j$. Let $E_j$ be the total space of the line bundles $L_j$ with the projection map $p_j: E_j \to M_j$ and $W_j$ denote the vector space $\{s:\Lambda_j \to E_j \;| \;p_j \circ s (z) =z \text{ for all } z\in \Lambda_j\}$ (that means $W_j$ is the space of sections restricted to $\Lambda_j$ but these sections do not have any additional structure like smoothness etc, since we have not assumed any structure on $\Lambda_j$ to keep our results in more general setup). Let  $\mathcal{R}_{\Lambda_j} : H^0(M_j,L_j) \to W_j$ be the linear map
$$s\mapsto s\Bigr |_{\Lambda_j}$$
that restricts a holomorphic section of $L_j\to M_j$ to the subset $\Lambda_j \subset M_j$. Similarly, let $E$ be the total space of line bundle $L_1\boxtimes L_2$ and $p: E\to M_1\times M_2$ be the projection map. For a subset $\Lambda$ of $M_1\times M_2$, let us denote $\mathcal{R}_{\Lambda} : H^0(M_1\times M_2,L_1 \boxtimes L_2) \to W$ be the map that restricts a holomorphic section of $L_1\boxtimes L_2 \to M_1\times M_2$ to the subset $\Lambda \subset M_1\times M_2$ where $W=\{s:\Lambda \to E \; | \; p \circ s (z) =z \text{ for all } z\in \Lambda\}$. We use the same notation $\lr_{\Lambda}$ to mean all these three restriction maps, but we pick the appropriate one based on whether $\Lambda$ is a subset of $M_1, M_2$ or $M_1\times M_2$. We hope that it does not create any confusion.
\begin{notation}
    For a subset $\Lambda \subset M_1 \times M_2$, let $\Pi_{\Lambda}^{\ker}$ and $\Pi_{\Lambda}^{\ker^{\perp}}$ be the orthogonal projections onto $\ker(\mathcal{R}_{\Lambda})$ and $\ker(\mathcal{R}_{\Lambda})^{\perp}$ respectively. %Then $\frac{1}{\tr(\Pi_{\Lambda}^{\ker})}\Pi_{\Lambda}^{\ker}$ and $\frac{1}{\tr(\Pi_{\Lambda}^{\ker^{\perp}})}\Pi_{\Lambda}^{\ker^{\perp}}$ are states, provided they are non-zero operators. 
    We denote $\rho_{\Lambda}^{\ker} =\frac{1}{\tr(\Pi_{\Lambda}^{\ker})}\Pi_{\Lambda}^{\ker} $ and $\rho_{\Lambda}^{\ker^{\perp}}= \frac{1}{\tr(\Pi_{\Lambda}^{\ker^{\perp}})}\Pi_{\Lambda}^{\ker^{\perp}}$ (whenever the trace is not zero). Whenever the trace is zero, we define the corresponding $\rho_{\Lambda}^{\ker}$ or $\rho_{\Lambda}^{\ker^{\perp}}$ as the zero operator.
\end{notation}
In this article, we associate the states $\rho_{\Lambda}^{\ker}$ and $\rho_{\Lambda}^{\ker^{\perp}}$ with the subset $\Lambda$ of $M_1 \times M_2$. Note that $\ker(\mathcal{R}_{\Lambda})$ or $\ker(\mathcal{R}_{\Lambda})^{\perp}$ can be zero sometimes and in that case the corresponding operator $\rho_{\Lambda}^{\ker}$ or $\rho_{\Lambda}^{\ker^{\perp}}$ is technically not a state, because of it being the zero operator, the trace condition on states is satisfied. However, at least one of $\rho_{\Lambda}^{\ker}$ or $\rho_{\Lambda}^{\ker^{\perp}}$ is a state (as their sum is $I$, the identity operator). In this article, we ignore this technicality and say that both of them are states, even when one of the operators is zero.
\begin{obs} \label{obs}
    We observe that $\rho_{\Lambda}^{\ker}+\rho_{\Lambda}^{\ker^{\perp}}=I$. Since their sum $I$ is a separable state, we can not have one of them separable while the other is not separable. Therefore $\rho_{\Lambda}^{\ker}$ is separable if and only if $\rho_{\Lambda}^{\ker^{\perp}}$ is separable.
\end{obs}

In \cite{barron2022entanglement}, the authors considered a submanifold $\Lambda$ of $M_1 \times M_2$ and considered the restriction map $\mathcal{R}_{\Lambda} : H^0(M_1\times M_2,L_1 \boxtimes L_2) \to L^2(\Lambda, L_1 \boxtimes L_2|_{\Lambda}
)$. In their setup, due to the additional smooth structure on $\Lambda$, the integrals over $\Lambda$ with respect to the measure $\mu$ are defined and in fact, the codomain becomes the space of $L^2$ sections. Using $\lr_{\Lambda}$ and the inner product on the codomain of this map, they associated certain states with $\Lambda$ and proved that when $\Lambda = \Lambda_1 \times \Lambda_2$ is a product submanifold then these states are separable. This motivated us to associate states a little differently and check if similar properties can be observed. Due to the use of the inner product on the codomain, they considered nice subsets over which we can integrate, such as submanifolds. Since the states $ \rho_{\Lambda}^{\ker}$ and $\rho_{\Lambda}^{\ker^{\perp}}$ we associated in our setup are directly linked to subspaces of $H^0(M_1\times M_2,L_1 \boxtimes L_2)$ and we are not using any additional structure of the co-domain other than it being a vector space, we don't restrict ourselves to submanifolds and take arbitrary subsets.\\

In an attempt to investigate similar properties of these newly associated states $ \rho_{\Lambda}^{\ker}$ and $\rho_{\Lambda}^{\ker^{\perp}}$ with respect to product submanifolds, we proof a result (proposition \ref{productKer}) that provides a nice description of the kernel of the restriction map to a product subset. Our main result theorem \eqref{union_thorem} extends proposition \eqref{productKer} to provide a nice description of the kernel of the restriction map to even more generality, to any finite union of products. As a corollary of \ref{union_thorem}, we get that our states are also separable for product submanifolds. This answers the similar result they got in \cite{barron2022entanglement}. In fact, we get a more general result that says the states are separable if $\Lambda$ is a finite union of products.\\

In section \ref{last_section}, we ask the question whether, for every state $\sigma$ on $\HAB$, there exists a subset $\Lambda$ of $M_1 \times M_2$ such that $\rho_{\Lambda}^{\ker}=\sigma$. As our constructed state is an orthogonal projection, clearly if $\sigma$ is not an orthogonal projection, then the answer is no. We partially answer the question for quantum states that are orthogonal projections. The solution is straightforward for pure states by the projective embedding of $M_1 \times M_2$ using the quantum line bundle. This also affirmatively says that the states $ \rho_{\Lambda}^{\ker}$ and $\rho_{\Lambda}^{\ker^{\perp}}$ associated this way with subsets are not always separable. For mixed states having a particular property, we use coherent states and the covariant symbol of $\sigma$ to construct the subset.

%.....................Section.................%%%%
%%%%%%%%%%%%%%%%%%%%%%%%%%%%%%%%%%%%%%%%%%%%%%%%%%%%%
\section{Main results}
The purpose of this section is to prove that $ \rho_{\Lambda}^{\ker}$ and $\rho_{\Lambda}^{\ker^{\perp}}$ are separable when $\Lambda \subset M_1\times M_2$ is a finite union of product. For that purpose, we show that the range space of the states contains an orthonormal basis consisting of only decomposable vectors. If $\Lambda$ is a union of products, then we present a description of $\ker(\lr_{\Lambda})$ as a direct sum of Hilbert spaces each of which is a tensor product of Hilbert spaces that are orthogonal to one another. We get this in a few steps. First, we consider $\Lambda= \Lambda_1\times\Lambda_2$ where one of $\Lambda_1$ or $\Lambda_2$ is a singleton set and prove proposition \eqref{pointSub}. We use proposition \eqref{pointSub} to generalize it for products without the restriction of $\Lambda_1$ or $\Lambda_2$ being singleton to get proposition \eqref{productKer}.

\begin{proposition} \label{pointSub}
    For $j\in \{1,2\}$, let $\Lambda_j$ be a non-empty subset of $M_j$. If either $\Lambda_1$ or $\Lambda_2$ is a singleton set, then $$\ker(\mathcal{R}_{\Lambda_1 \times \Lambda_2}) = A_1 \oplus A_2,$$
    where 
    \begin{align*}
        A_1 = & \HA \otimes \ker(\mathcal{R}_{\Lambda_2}) \\
        A_2 = & \ker(\mathcal{R}_{\Lambda_1})\otimes \ker(\mathcal{R}_{\Lambda_2})^{\perp}.
    \end{align*}
    We dedude that $\ker(\mathcal{R}_{\Lambda_1 \times \Lambda_2})^{\perp} = \ker(\mathcal{R}_{\Lambda_1})^{\perp} \otimes \ker(\mathcal{R}_{\Lambda_2})^{\perp}$. 
\end{proposition}
\begin{proof}
    Suppose $\Lambda_2 = \{y\}$. Let $(x,y) \in \Lambda_1 \times \Lambda_2$ and $s \in  [\HA \otimes \ker(\mathcal{R}_{\Lambda_2})] \oplus [\ker(\mathcal{R}_{\Lambda_1})\otimes (\ker(\mathcal{R}_{\Lambda_2}))^{\perp}]$. Then for some $n_1,n_2 \ge 1$, we can write $$s= \sum_{j=1}^{n_1} u_j\otimes v_j + \sum_{j=1}^{n_2}u'_j \otimes v'_j$$ where $v_j \in \ker(\lr_{\Lambda_2})$ for all $j\in \{1,...,n_1\}$ and $u_j' \in \ker(\lr_{\Lambda_1})$ for all $j \in \{1,...,n_2\}$. Take a trivializing open cover of the line bundle $L_1 \boxtimes L_2 \to M_1 \times M_2$. Then in an open neighbourhood of $(x,y)$ in the trivializing cover, we get that  $v_j(y)=0$ for all $j\in \{1,...,n_1\}$ and $u_j'(x)=0 $ for all $j \in \{1,...,n_2\}$.  Therefore, 
    \begin{align*}
        s(x,y) & = \sum_{j=1}^{n_1} u_j(x)\otimes v_j(y) + \sum_{j=1}^{n_2}u'_j(x) \otimes v'_j(y) =0.
    \end{align*}
    We get that $s \in \ker(\mathcal{R}_{\Lambda_1 \times \Lambda_2})$. Therefore, $A_1\oplus A_2 \subset \ker(\mathcal{R}_{\Lambda_1 \times \Lambda_2})$.\\
    
    For the inclusion from the other side, let $s \in \ker(\mathcal{R}_{\Lambda_1 \times \Lambda_2})$. Then there exists $d\in \N$ such that we can write $s = \sum_{j=1}^d u_j\otimes v_j$ where $u_j \in \HA$ and $v_j\in \HB$ are non-zero.
    If for some $j\in \{1,...,d\}$, we have $v_j\in \ker(\mathcal{R}_{\Lambda_2})$, then $u_j\otimes v_j \in \HA \otimes \ker(\mathcal{R}_{\Lambda_2})$. So we assume that $v_j \in (\ker(\mathcal{R}_{\Lambda_2}))^{\perp}$ for all $j \in \{1,...,d\}$, therefore $v_j(y) \ne 0$ for all $j \in  \{1,...,d\}$.\\

    Take a trivializing open cover of the line bundle $L_2 \to M_2$. In an open neighbourhood of $y$ in the trivializing open cover, each of the sections $v_j$ is given by a map with co-domain $\C$. Therefore $\{v_j(y):j\in \{1,2,...,d\}\}$ is a set of non-zero complex numbers. There exists non-zero $\lambda_2,...,\lambda_d \in \C$ such that $v_j(y) = \lambda_j v_1(y)$ for $j\in \{2,...,d\}$, i.e. for each $j$, we have $v_j - \lambda_j v_1 \in \ker(\mathcal{R}_{\Lambda_2})$. Now, let $x\in \Lambda_1$. Take a trivializing open cover of the line bundle $L_1 \to M_1$. In an open neighbourhood of $x$ in the trivializing open cover, each of the sections $u_j$ is given by a map with co-domain $\C$. By assumption, the section $s$ vanishes at $(x,y)$, therefore we have
    \begin{align*}
        & \sum_{j=1}^d u_j(x)v_j(y)=0 \\
        \Rightarrow \; & u_1(x)+\sum_{j=2}^d\lambda_ju_j(x)=0 \\
        \Rightarrow \; &  u_1+\sum_{j=2}^d\lambda_ju_j \in \ker(\mathcal{R}_{\Lambda_1})
    \end{align*}
    We can rearrange the terms to write $s$ in the following way: $$s=\sum_{j=1}^{d}u_j \otimes v_j = (u_1+\sum_{j=2}^d\lambda_ju_j)\otimes v_1 + \sum_{j=2}^su_j\otimes (v_j-\lambda_j v_1).$$ 
    Note that the first term in the above expression is an element of  $\ker(\mathcal{R}_{\Lambda_1})\otimes (\ker(\mathcal{R}_{\Lambda_2}))^{\perp} $ and the second term is an element of $\HA \otimes \ker(\mathcal{R}_{\Lambda_2})$. This finishes the proof of $\ker(\mathcal{R}_{\Lambda_1 \times \Lambda_2}) \subset A_1 \oplus A_2$. Hence $\ker(\mathcal{R}_{\Lambda_1 \times \Lambda_2}) =[\ker(\mathcal{R}_{\Lambda_1})\otimes (\ker(\mathcal{R}_{\Lambda_2}))^{\perp}] \oplus [\HA \otimes \ker(\mathcal{R}_{\Lambda_2})]$.\\
    
    When $\Lambda_1$ is a singleton set, we can use a similar argument (with roles of $\Lambda_1$ and $\Lambda_2$ reversed) to show that $\ker(\mathcal{R}_{\Lambda_1 \times \Lambda_2}) = [\ker(\mathcal{R}_{\Lambda_1})\otimes \HB] \oplus [(\ker(\mathcal{R}_{\Lambda_1}))^{\perp} \otimes \ker(\mathcal{R}_{\Lambda_2})]$ which is equal to the space $[\HA \otimes \ker(\mathcal{R}_{\Lambda_2})] \oplus [\ker(\mathcal{R}_{\Lambda_1})\otimes (\ker(\mathcal{R}_{\Lambda_2}))^{\perp}]$.\\

    Finally, we note that $A_1 \oplus A_2 \oplus \left[\ker(\mathcal{R}_{\Lambda_1})^{\perp} \otimes \ker(\mathcal{R}_{\Lambda_2})^{\perp}\right] = \HAB$ and 
    $\langle u,v \rangle =0$ for any $u\in \ker(\mathcal{R}_{\Lambda_1})^{\perp} \otimes \ker(\mathcal{R}_{\Lambda_2})^{\perp} $ and $v\in H_1 \oplus H_2 $, therefore $\ker(\mathcal{R}_{\Lambda_1 \times \Lambda_2})^{\perp} = \ker(\mathcal{R}_{\Lambda_1})^{\perp} \otimes \ker(\mathcal{R}_{\Lambda_2})^{\perp}$.
\end{proof}

Now that we have our description when one of the subsets involved is a point, our idea is to write $\Lambda_1\times \Lambda_2$ as the union $\bigcup_{x \in\Lambda_1} \{x\}\times \Lambda_2$, so that we can use the above to get $\ker(\lr_{\Lambda_1\times \Lambda_2})$ as intersections of subspaces. This allows us to extend the result to arbitrary $\Lambda_1$ and $\Lambda_2$. We keep in mind that the index set over which we have the intersection is an arbitrary index set and we need the direct sum to distribute over the intersection. This is not always true, not even for a finite index set (as an example take $V=\C^2$, $H_1 = $ span$\{(1,0)\},$ $H_2 = $ span$\{(0,1)\}$ and $H_3 = $ span$\{(1,1)\}$, then $(H_1 \oplus H_2) \cap (H_1 \oplus H_3) = \C^2$ but $H_1 \oplus (H_2 \cap H_3)=H_1 \ne \C^2$). However, if we put further conditions, which are satisfied in our case, then the direct sum distributes over the intersection. We provide a short proof of this fact in the appendix. The condition to make this distributive property hold is a crucial part that allows us to extend proposition \eqref{pointSub}. Further, we also need the distributive property of tensor products over intersections, which is true for vector spaces. We state these two lemmas here.
\begin{lemma} \label{sumDistIntersec}
    Let $\mathcal{I}$ be an arbitrary index set. Let $\{V_j\}_{j\in \mathcal{I}}$ and $\{W_j\}_{j\in \mathcal{I}}$ be two sets of subspaces of a vector space $X$. Let there exist two subspaces $V$ and $W$ of $X$ such that $\sum_{j\in \mathcal{I}}V_j \subset V$, $\sum_{j\in \mathcal{I}}W_j \subset W$ and $W \cap V = \{0\}$, then $$\bigcap_{j\in \mathcal{I}}(W_j\oplus V_j) = \left(\bigcap_{j\in \mathcal{I}}W_j\right) \oplus \left(\bigcap_{j\in \mathcal{I}}V_j\right).$$
\end{lemma}
\begin{lemma}\label{tensorDistIntersec} 
%Tensor distributes over intersection
    Let $\mathcal{I}$ be an arbitrary index set and $\{V_j\}_{j\in \mathcal{I}}$ be a set of subspaces of a vector space $V$. For an arbitrary vector space $W$, we have $$\bigcap_{j\in \mathcal{I}} (V_j \otimes W) = \left(\bigcap_{j\in \mathcal{I}} V_j\right) \otimes W.$$
\end{lemma}

Now let us use proposition \eqref{pointSub} and the above two lemmas to prove the following proposition.

\begin{proposition}\label{productKer}
     For $j\in \{1,2\}$, let $\Lambda_j$ be a non-empty subset of $M_j$. Then, $$\ker(\mathcal{R}_{\Lambda_1 \times \Lambda_2}) = A_1 \oplus A_2,$$
    where 
    \begin{align*}
        A_1 = & \HA \otimes \ker(\mathcal{R}_{\Lambda_2}) \\
        A_2 = & \ker(\mathcal{R}_{\Lambda_1})\otimes (\ker(\mathcal{R}_{\Lambda_2}))^{\perp}.
    \end{align*}
    We dedude that $\ker(\mathcal{R}_{\Lambda_1 \times \Lambda_2})^{\perp} = \ker(\mathcal{R}_{\Lambda_1})^{\perp} \otimes \ker(\mathcal{R}_{\Lambda_2})^{\perp}$.   
\end{proposition}
\begin{proof}
    Let $x \in \Lambda_1$. Using proposition \eqref{pointSub}, we have $$\ker(\mathcal{R}_{ \{x\}\times \Lambda_2}) = \left[\HA \otimes \ker(\mathcal{R}_{\Lambda_2})\right] \oplus \left[\ker(\mathcal{R}_{\{x\}})\otimes (\ker(\mathcal{R}_{\Lambda_2}))^{\perp}\right].$$
    For any $s\in \HAB$,
    \begin{align*}
        & s \in \ker(\mathcal{R}_{\Lambda_1 \times \Lambda_2}) \\
        \iff & s (x,y) = 0 \; \text{ for all } x\in \Lambda_1, y\in \Lambda_2 \\
        \iff & s \in \ker(\mathcal{R}_{\{x\} \times \Lambda_2}) \; \text{ for all } x\in \Lambda_1\\
        \iff & s \in \bigcap_{x\in \Lambda_1} \ker(\mathcal{R}_{\{x\} \times \Lambda_2 })
    \end{align*}
    Therefore, $ \ker(\mathcal{R}_{\Lambda_1 \times \Lambda_2}) = \bigcap\limits_{x\in \Lambda_1} \ker(\mathcal{R}_{\{x\} \times \Lambda_2 })$. Similarly, it is easy to verify that $\ker(\mathcal{R}_{\Lambda_1}) = \bigcap\limits_{x \in \Lambda_1} \ker(\mathcal{R}_{\{x\}})$. We notice that $$\sum_{x\in \Lambda_1}\left(\ker(\mathcal{R}_{\{x\}}) \otimes (\ker(\mathcal{R}_{\Lambda_2}))^{\perp}\right) \subset \HA \otimes (\ker(\mathcal{R}_{\Lambda_2}))^{\perp} $$ and $$ \left(\HA \otimes \ker(\mathcal{R}_{\Lambda_2})\right) \cap \left(\HA \otimes (\ker(\mathcal{R}_{\Lambda_2}))^{\perp}\right) = \{0\},$$so lemma \ref{sumDistIntersec} can be applied in the following.  We have,
    \begin{align*}
        & \ker(\mathcal{R}_{\Lambda_1 \times \Lambda_2}) \\
        = & \bigcap\limits_{x\in \Lambda_1} \ker(\mathcal{R}_{\{x\} \times \Lambda_2 }) \\
        = & \bigcap_{x\in \Lambda_1}\left[\left(\HA \otimes \ker(\mathcal{R}_{\Lambda_2})\right) \oplus \left(\ker(\mathcal{R}_{\Lambda_1})\otimes (\ker(\mathcal{R}_{\Lambda_2}))^{\perp}\right)\right]\\
        = & \left[\HA \otimes \ker(\mathcal{R}_{\Lambda_2})\right] \oplus \left[\bigcap_{x\in \Lambda_1}\left(\ker(\mathcal{R}_{\{x\}})\otimes (\ker(\mathcal{R}_{\Lambda_2}))^{\perp}\right)\right] \text{\quad (using lemma \eqref{sumDistIntersec})}\\
        = & \left[\HA \otimes \ker(\mathcal{R}_{\Lambda_2})\right] \oplus \left[\left(\bigcap_{x\in \Lambda_1}\ker(\mathcal{R}_{\{x\}})\right)\otimes (\ker(\mathcal{R}_{\Lambda_2}))^{\perp}\right] \text{\quad (using lemma \eqref{tensorDistIntersec})}\\
        = & \left[\HA \otimes \ker(\mathcal{R}_{\Lambda_2})\right] \oplus \left[\ker(\mathcal{R}_{\Lambda_1})\otimes (\ker(\mathcal{R}_{\Lambda_2}))^{\perp}\right]
    \end{align*}

\end{proof}
This leads our way to prove the main theorem in this article. We try to recognize the pattern in the product case to extend it to a finite union of products. For this purpose, if we have $\Lambda = \bigcup_{j\in [n]}(U_j \times V_j)$, we apply proposition \eqref{productKer} to decompose the kernels of each of the products $U_j \times V_j$ in the union cleverly to arrive at a nice direct sum decomposition indexed by the subsets of $\{1,2,...,n\}$.

\begin{theorem} \label{union_thorem}
    Let $U_1, U_2,...,U_n$ be non-empty subsets of $M_1$ and $V_1, V_2,...,V_n$ be non-empty subsets of $M_2$. Denote $[n] = \{1,2,,...,n\}$. Let $\Lambda = \bigcup_{j\in [n]}(U_j \times V_j)$.  Then
    $$\ker(\mathcal{R}_{\Lambda}) = \bigoplus_{S \subset [n]} \left(H^{(n)}_S \otimes K^{(n)}_S\right)$$
    where
    \[
    H^{(n)}_S = \begin{cases}
    \HA & \text{\quad if } S=\phi\\
    \bigcap\limits_{j\in S}\ker(\lr_{U_j}) & \text{\quad otherwise}
    \end{cases}
    \]
    and $K^{(n)}_S = \bigcap_{j=1}^n W_{S,j}$ where
    \[
    W_{S,j} = \begin{cases}
        \ker(\lr_{V_j})^{\perp} & \text{\quad if } j\in S \\
        \ker(\lr_{V_j}) & \text{\quad if } j\notin S.
    \end{cases}
    \]
\end{theorem}
\begin{proof}
    We shall use induction on $n$. The base case $n=1$ is true by proposition \eqref{productKer}.\\
    Suppose the statement is true for some $n\in \N$. Denote $\Lambda' = \bigcup_{j\in [n]}(U_j \times V_j)$ and $\Lambda = \bigcup_{j\in [n+1]}(U_j \times V_j)$. Then by the induction hypothesis, we have
    \begin{equation}\label{hypothesis}
    \ker(\lr_{\Lambda'}) = \bigoplus_{S \subset [n]} \left(H^{(n)}_S \otimes K^{(n)}_S\right).  
    \end{equation}
    Using proposition \eqref{productKer}, we have 
    \begin{equation} \label{n+1thterm}
    \ker(\lr_{U_{n+1}\times V_{n+1}}) = \left[\HA \otimes \ker(\lr_{V_{n+1}})\right] \oplus \left[\ker(\lr_{U_{n+1}}) \otimes (\ker(\lr_{V_{n+1}}))^{\perp}\right].   
    \end{equation}
    Since $\Lambda = \Lambda' \cup (U_{n+1} \times V_{n+1})$, we get that
    \begin{equation} \label{intersectionequation}
    \ker(\lr_{\Lambda}) = \ker(\lr_{\Lambda'}) \cap \ker(\lr_{U_{n+1}\times V_{n+1}}).    
    \end{equation}
    Note that if $W$ and $W_1,...,W_k$ are subspaces of a Hilbert space $H$ such that $W_1,...,W_k \subset W$, then 
    \begin{equation*}
     W =  (W_1+...+W_k) \oplus W_{1,...,k}' \text{ for some } W'_{1,...,k} \subset W. 
    \end{equation*}
    Note that there are infinitely many $W'_{1,...,k}$ that satisfy the above properties. We choose and fix one of them to define the following subspaces $X_S,Y_S,Z_S,P$ and $Q$ of $\HAB$ having the property:
    \begin{align*}
        H_{S}^{(n+1)} & = H_{S\cup\{n+1\}}^{(n+1)} \oplus X_S \\
        K_{S}^{(n)} & = K_S^{(n+1)} \oplus K_{S\cup\{n+1\}}^{(n+1)} \oplus Y_S \\
        H_{\{n+1\}}^{(n+1)} & = H_{S\cup \{n+1\}}^{(n+1)} \oplus Z_S
        \text{\quad for each subset $S \subset [n]$},
    \end{align*}
    and 
    \begin{align*}
        \ker(\lr_{V_{n+1}})& = \left(\bigoplus_{S \subset [n]} K_S^{(n+1)}\right) \oplus P\\
        \ker(\lr_{V_{n+1}})^{\perp}& = \left(\bigoplus_{S \subset [n]} K_{S\cup \{n+1\}}^{(n+1)}\right) \oplus Q. 
    \end{align*}
    
If $S_1, S_2 \subset [n+1]$ with $S_1 \ne S_2$, then $(S_1 \setminus S_2)$ or $(S_2 \setminus S_1)$ is non-empty. Say $(S_1 \setminus S_2)$ is non-empty, then there exists $j \in [n+1]$ such that $j \in S_1 \setminus S_2$, so $K_{S_1}^{(n+1)} \subset \ker(\lr_{V_j})^{\perp}$ and $ K_{S_2}^{(n+1)} \subset \ker(\lr_{V_j})$, i.e. $K_{S_1}^{(n+1)}$ and $K_{S_2}^{(n+1)}$ are orthogonal to each other and we get that $ K_{S_1}^{(n+1)} \cap K_{S_2}^{(n+1)} = \{0\}$. That is why we can write direct sum instead of sum in the equations involving $P$ and $Q$.\\
    
Also note that for any $S \subset [n]$, if $v \in X_S \cap H_{\{n+1\}}^{(n+1)}$, then $ v\in H_{S}^{(n+1)} \cap   H_{\{n+1\}}^{(n+1)} =  H_{S\cup\{n+1\}}^{(n+1)}$, i.e. $v \in X_S \cap H_{S\cup\{n+1\}}^{(n+1)} = \{0\} $, i.e. $v=0$. Therefore  we get that $X_S \cap H_{\{n+1\}}^{(n+1)}= \{0\}$ and $X_S \cap Z_S = \{0\}$.\\
    
    Note that if $S \subset [n]$, then $H_S^{(n)} = H_S^{(n+1)}$, therefore for each $S \subset [n]$, we can have the decomposition
    \begin{align} \label{S_term}
        \nonumber& H^{(n)}_S \otimes K^{(n)}_S \\
    \nonumber    = & H^{(n+1)}_S \otimes \left[K_S^{(n+1)} \oplus K_{S\cup\{n+1\}}^{(n+1)} \oplus Y_S\right]  \\
    \nonumber    = & \left[H^{(n+1)}_S \otimes K_S^{(n+1)}\right] \oplus  \left[(H_{S\cup\{n+1\}}^{(n+1)} \oplus X_S) \otimes K_{S\cup\{n+1\}}^{(n+1)}\right] \oplus \left[ H^{(n+1)}_S \otimes Y_S\right]  \\
        = & \left[H^{(n+1)}_S \otimes K^{(n+1)}_S\right] \oplus \left[H^{(n+1)}_{S\cup \{n+1\}} \otimes K^{(n+1)}_{S\cup \{n+1\}}\right] \oplus \left[X_S \otimes K^{(n+1)}_{S\cup \{n+1\}}\right] \oplus \left[H^{(n+1)}_S \otimes Y_S \right]   
    \end{align}
    Now, we can decompose the two subspaces appearing in $\ker(\lr_{U_{n+1}\times V_{n+1}})$ as follows
    \begin{align} \label{1st_term_of_second}
       \nonumber & \HA \otimes \ker(\lr_{V_{n+1}}) \\
    \nonumber = & \HA \otimes \left[\left(\oplus_{S \subset [n]} K_S^{(n+1)}\right)\oplus P \right]\\
    \nonumber= & \left[H^{(n+1)}_{\phi} \otimes K^{(n+1)}_{\phi}\right] \oplus \left[ \bigoplus_{\substack{S\subset [n]\\S\ne \phi}} \left(H^{(n+1)}_{S} \otimes K^{(n+1)}_{S}\right) \right] \\
    & \hspace{4cm}\oplus \left[ \bigoplus_{\substack{S\subset [n]\\S\ne \phi}} \left( {H^{(n+1)}_S}^{\perp} \otimes K^{(n+1)}_{S} \right) \right]  \oplus [\HA \otimes P]
    \end{align}
    and
    \begin{align}\label{2nd_term_of_second}
       \nonumber & \ker(\lr_{U_{n+1}}) \otimes (\ker(\lr_{V_{n+1}}))^{\perp} \\
    \nonumber = & H_{\{n+1\}}^{(n+1)} \otimes  \left[\left(\oplus_{S \subset [n]} K_{S\cup \{n+1\} }^{(n+1)}\right) \oplus Q \right] \\
    \nonumber = & \left[H_{\{n+1\}}^{(n+1)} \otimes K^{(n+1)}_{\{n+1\}}\right] \oplus \left[ \bigoplus_{\substack{S\subset [n]\\S\ne \phi}} \left(H^{(n+1)}_{S\cup \{n+1\} } \otimes K^{(n+1)}_{S \cup \{n+1\}} \right)\right]  \\
    & \hspace{4cm}\oplus \left[ \bigoplus_{\substack{S\subset [n]\\S\ne \phi}} \left( Z_S \otimes K^{(n+1)}_{S\cup \{n+1\}}\right) \right]
     \oplus \left[H_{\{n+1\}}^{(n+1)}\otimes Q\right] 
    \end{align}
    Let  $S_0,S \subset [n]$.
    \begin{itemize}
        \item Since $S \ne S_0 \cup \{n+1\}$ and so $ K^{(n+1)}_{S_0\cup \{n+1\}} $ and $K^{(n+1)}_{S}$ are orthogonal, we have $$\left[X_{S_0} \otimes K^{(n+1)}_{S_0\cup \{n+1\}}\right] \cap \left[{H^{(n+1)}_S}^{\perp} \otimes K^{(n+1)}_{S}\right] = \{0\}.$$ 
        \item Since $K^{(n+1)}_{S_0\cup \{n+1\}} \subset \ker(\lr_{V_{n+1}})^{\perp}$ and $P \subset \ker(\lr_{V_{n+1}})$, we have $$\left[X_{S_0} \otimes K^{(n+1)}_{S_0\cup \{n+1\}}\right] \cap [\HA \otimes P] = \{0\}.$$
        \item When $S \ne S_0$ then $K^{(n+1)}_{S\cup \{n+1\}}$ and $K^{(n+1)}_{S_0\cup \{n+1\}}$ are orthogonal and when $S=S_0$ then $X_{S_0} \cap Z_{S_0} = \{0\}$, therefore we have $$\left[X_{S_0} \otimes K^{(n+1)}_{S_0\cup \{n+1\}}\right]  \cap \left[ Z_S \otimes K^{(n+1)}_{S\cup \{n+1\}} \right] = \{0\}.$$
        \item Since $X_{S_0} \cap H_{\{n+1\}}^{(n+1)}= \{0\}$, we have $$\left[X_{S_0} \otimes K^{(n+1)}_{S_0\cup \{n+1\}}\right]  \cap \left[ H_{\{n+1\}}^{(n+1)}\otimes Q \right] = \{0\}.$$ 

        \item When $S=S_0$, then $Y_{S_0} \cap K^{(n+1)}_{S_0} = \{0\}$ by definition of $Y_{S_0}$. When $S \ne S_0$ then either $S\setminus S_0$ or $S_0\setminus S$ is non-empty. Say $S \setminus S_0$ is non-empty, then there exists $j\in S\setminus S_0$ with $j\in [n]$ so that $K^{(n+1)}_{S} \subset
        \ker(\lr_{V_j})^{\perp}$ and $ Y_{S_0} \subset \ker(\lr_{V_j})$ and so $Y_{S_0} \cap K^{(n+1)}_{S} = \{0\}$. Therefore, we have $$\left[H^{(n+1)}_{S_0} \otimes Y_{S_0}\right] \cap \left[{H^{(n+1)}_{S}}^{\perp} \otimes K^{(n+1)}_{S}\right] = \{0\}.$$ 
        \item If $ v\in Y_{S_0} \cap P$, then $v \in K_{S_0} \cap \ker(\lr_{V_{n+1}}) = K^{(n+1)}_{S_0}$, i.e. $v \in Y_{S_0} \cap K^{(n+1)}_{S_0}=\{0\}$. Therefore, we have $$\left[H^{(n+1)}_{S_0} \otimes Y_{S_0}\right] \cap [\HA \otimes P] = \{0\}.$$ 
        \item When $S=S_0$, then $Y_{S_0} \cap K^{(n+1)}_{S_0 \cup \{n+1\} } = \{0\}$ by definition of $Y_{S_0}$. When $S \ne S_0$ then either $S\setminus S_0$ or $S_0\setminus S$ is non-empty. Say $S \setminus S_0$ is non-empty, then there exists $j\in S\setminus S_0$ with $j\in [n]$ so that $K^{(n+1)}_{S\cup \{n+1\}} \subset
        \ker(\lr_{V_j})^{\perp}$ and $ Y_{S_0} \subset \ker(\lr_{V_j})$ and so $Y_{S_0} \cap K^{(n+1)}_{S\{n+1\}} = \{0\}$. Therefore, we have $$\left[H^{(n+1)}_{S_0} \otimes Y_{S_0}\right]  \cap \left[ Z_S \otimes K^{(n+1)}_{S\cup \{n+1\}} \right] = \{0\}.$$ 
        \item If $ v\in Y_{S_0} \cap Q$, then $v \in K_{S_0} \cap \ker(\lr_{V_{n+1}})^{\perp} = K^{(n+1)}_{S_0\cup \{n+1\}}$, i.e. $v \in Y_{S_0} \cap K^{(n+1)}_{S_0\cup \{n+1\}}=\{0\}$. Therefore, we have $$\left[H^{(n+1)}_{S_0} \otimes Y_{S_0}\right]  \cap \left[ H_{\{n+1\}}^{(n+1)}\otimes Q \right] = \{0\} .$$ 
    \end{itemize}
    Therefore, using these information and equations \eqref{hypothesis}, \eqref{n+1thterm}, \eqref{intersectionequation}, \eqref{S_term}, \eqref{1st_term_of_second} and \eqref{2nd_term_of_second}, we see that we can write
    $$   \ker(\lr_{\Lambda})
       = \left[ \left(\bigoplus_{S \subset [n+1]} \left(H^{(n+1)}_S \otimes K^{(n+1)}_S\right)\right) \oplus F \right] \cap \left[ \left(\bigoplus_{S \subset [n+1]} \left(H^{(n+1)}_S \otimes K^{(n+1)}_S\right)\right) \oplus F' \right]$$
    where 
    \begin{align*}
    F & = \bigoplus_{S \subset [n]} \left( \left[X_S \otimes K^{(n+1)}_{S\cup \{n+1\}}\right] \oplus \left[H^{(n+1)}_S \otimes Y_S \right]\right), \\
    F' & = \left[ \bigoplus_{\substack{S\subset [n]\\S\ne \phi}} \left( {H^{(n+1)}_S}^{\perp} \otimes K^{(n+1)}_{S} \right) \right]  \oplus [\HA \otimes P] \\
    & \hspace{5cm}\oplus \left[ \bigoplus_{\substack{S\subset [n]\\S\ne \phi}} \left( Z_S \otimes K^{(n+1)}_{S\cup \{n+1\}}\right) \right]
     \oplus \left[H_{\{n+1\}}^{(n+1)}\otimes Q\right] 
    \end{align*}
    with the property that $F \cap F' = \{0\}$, which implies
    $$\ker(\lr_{\Lambda})
       =  \bigoplus_{S \subset [n+1]} \left(H^{(n+1)}_S \otimes K^{(n+1)}_S\right).$$
    Hence we are done with the induction proof.
\end{proof}

\begin{example}
    To see the decomposition better, we provide the subspaces involved for the case $n=2$ here:
    \begin{align*}
        H^{(2)}_{\phi} \otimes K^{(2)}_{\phi} & =  \HA \otimes [\ker(\lr_{V_1}) \cap \ker(\lr_{V_2})] \\
        H^{(2)}_{ \{1\} } \otimes K^{(2)}_{\{1\}} & = \ker(\lr_{U_1}) \otimes \left[\ker(\lr_{V_1})^{\perp} \cap \ker(\lr_{V_2})\right] \\
        H^{(2)}_{\{2\}} \otimes K^{(2)}_{\{2\}} & = \ker(\lr_{U_2}) \otimes \left[\ker(\lr_{V_1}) \cap \ker(\lr_{V_2})^{\perp}\right]\\
        H^{(2)}_{\{1,2\}} \otimes K^{(2)}_{\{1,2\}} & = \left[\ker(\lr_{U_1}) \cap \ker(\lr_{U_2})\right] \otimes \left[\ker(\lr_{V_1})^{\perp}  \cap \ker(\lr_{V_2})^{\perp}\right].
    \end{align*}  
\end{example}

\begin{corollary}\label{main_corollary}
    Let $\Lambda \subset M_1 \times M_2$ such that $\Lambda$ is the union of finitely many non-empty subsets of $M_1\times M_2$ all of which are products, i.e. $\Lambda = \bigcup_{j=1}^n (U_j\times V_j)$ for $U_j \subset M_1$ and $V_j \subset M_2$, then $\rho_{\Lambda}^{\ker}$ and $\rho_{\Lambda}^{\ker^{\perp}}$ are separable. In particular, when $\Lambda$ is a non-empty product subset, then $\rho_{\Lambda}^{\ker}$ and $\rho_{\Lambda}^{\ker^{\perp}}$ are separable.
\end{corollary}
\begin{proof}
    Using theorem \eqref{union_thorem}, we see that when $\Lambda = \bigcup_{j=1}^n (U_j \times V_j)$ then $\ker(\lr_{\Lambda})$ is direct sum of tensor product of Hilbert spaces. Further notice that for $S_1, S_2 \subset [n]$ with $S_1 \ne S_2$, the subspaces $H_{S_1}^{(n)} \otimes K_{S_1}^{(n)}$ and $H_{S_2}^{(n)} \otimes K_{S_2}^{(n)}$ are orthogonal to each other. This can be seen as $K_{S_1}^{(n)}$ and $K_{S_2}^{(n)}$ are orthogonal to each other when $S_1 \ne S_2$ as mentioned in the proof of theorem \eqref{union_thorem}. Therefore $\ker(\lr_{\Lambda})$ contains an orthonormal basis consisting of separable vectors. Hence, $\rho_{\Lambda}^{\ker}$, being the orthogonal projection onto the subspace $\ker(\lr_{\Lambda})$, can be written as a convex sum of the pure states in this orthonormal basis. But all these pure states in this orthonormal basis are separable, hence $\rho_{\Lambda}^{\ker}$ is separable. From observation \eqref{obs}, we get that $\rho_{\Lambda}^{\ker^{\perp}}$ is also separable.
\end{proof}

\begin{comment}

\begin{example}
The states $\rho^{\ker}_{\Lambda}$ and $\rho^{\ker^{\perp}}_{\Lambda}$ are separable for any polar divisor at a point $(x_1,x_2) \in M_1\times M_2$.  Definitions and references here......  polar divisor of $(x,y)$ look like $(\Sigma_x \times M_2) \cup (M_1 \times \Sigma_y) $. Then applying proposition \eqref{union}, we see that $\ker(R_{\Sigma_{(x,y)}}) = \ker(R_{\Sigma_x}) \otimes \ker(R_{\Sigma_y})$.
 
\end{example}
\end{comment}

\section{Subsets coming from states} \label{last_section}
In this section, we see that the states associated with subsets this way are not always separable. As a consequence of the following theorem, if we start with a holomorphic section $s_0$ that is itself not separable, to begin with, then the corresponding state $\rho_{\mathcal{V}(s_0)}^{\ker}$ is not separable. This gives us a family of submanifolds with the property that the associated states are not separable. 

\begin{theorem} \label{Vanishing_of_section}
    Let $s_0 \in \HAB$ be a pure state and $\mathcal{V}(s_0)=\{x\in M_1 \times M_2: s_0(x)=0\}$ be the zero set of $s_0$, then $\rho_{\mathcal{V}(s_0)}^{\ker} = \mathcal{P}_{s_0}$, i.e. $\rho_{\mathcal{V}(s_0)}^{\ker}$ is a pure state and as pure state it is equal to $s_0$. We deduce that the states $\rho_{\mathcal{V}(s_0)}^{\ker}$ and $\rho_{\mathcal{V}(s_0)}^{\ker^{\perp}}$ are separable if and only if $s_0$ is separable.
\end{theorem}
\begin{proof}
    As $L_1 \to M_1$ and $L_2 \to M_2$ are very ample quantum line bundles, therefore $L_1 \boxtimes L_2 \to M_1 \times M_2$ is also very ample quantum line bundle. We know that using this line bundle we can embed $\iota: M_1 \times M_2 \to \Pro^N $ for some $N\in \N$ and that $L_1 \boxtimes L_2 \cong \iota^*(\mathcal{O}(1))$ (as holomorphic line bundle), the pullback of the hyperplane line bundle on $\Pro^N$ \cite[][page 7]{Berceanu_2000}. Let $s \in \ker(\lr_{\mathcal{V}_{s_0}})$. There exist holomorphic sections $ \tau_0$ and $\tau$ of $\mathcal{O}(1) \to \Pro^N$ such that $s_0 = \iota^*(\tau_0)$ and $s = \iota^*(\tau)$. Let $\{(U_{\alpha},\phi_{U_{\alpha}})\}$ be a trivializing open cover of the line bundle $L_1 \boxtimes L_2 \to M_1 \times M_2$. Let $f_{0\alpha}$ and $f_{\alpha}$ be the local representing functions on $U_{\alpha}$ that determines the sections $s_0$ and $s$. Since $s_0$ and $s$ are pullbacks of sections $\tau_0$ and $\tau$ of $\mathcal{O}(1)$, therefore the order of vanishing of all the zeroes of $f_{0\alpha}$ and $f_{\alpha}$ is 1.\\

    For each $\alpha$, we define $h_{\alpha} : U_{\alpha} \to \C$ given by $h_{\alpha}(x) = \frac{f_{\alpha}(x)}{f_{0\alpha}(x)}.$
    Since the order of vanishing of both $f_{0\alpha}$ and $f_{\alpha}$ is 1 and the vanishing set of $f_{0\alpha}$ is a subset of the vanishing set of $f_{\alpha}$, therefore $h_{\alpha}$ is a holomorphic function on $U_{\alpha}$. Also for $\alpha$ and $\beta$ with $U_{\alpha} \cap U_{\beta} \ne \phi$ we have
    $$h_{\alpha}(x) =\frac{f_{\alpha}(x)}{f_{0\alpha}(x)}= \frac{g_{\alpha\beta}(x)f_\beta(x)}{g_{\alpha\beta}(x)f_\beta(x)} = \frac{f_{\beta}(x)}{f_{0\beta}(x)}= h_{\beta}(x) \text{ \quad for all } x \in U_{\alpha} \cap U_{\beta}$$
    where $g_{\alpha\beta}$ is the transition function (so $g_{\alpha\beta}(x) \in GL_1(\C)$ is non-zero). Therefore $h$ given by $h_{\alpha}$ on $U_{\alpha}$ is a holomorphic function on $M_1 \times M_2$. But since $M_1 \times M_2$ is compact and connected, we see that $h$ is constant. Therefore, we get that $s$ is a constant multiple of $s_0$. Hence we have,
    $$ \ker(\lr_{\mathcal{V}_{s_0}}) = \{\lambda s_0: \lambda \in \C\} \text{ and } \rho_{\mathcal{V}(s_0)}^{\ker} = \mathcal{P}_{s_0}.$$
\end{proof}
In the previous theorem, we considered a particular subset that is related to the pure state $s_0$ and showed that the state $\rho_{\mathcal{V}(s_0)}^{\ker}$ is the same as the original pure state to begin with. We want to find special subsets that have this property with respect to a mixed state. As our associated states are orthogonal projections, we note that for a mixed state to have a corresponding subset, it has to be an orthogonal projection. We define these subsets using coherent states and the covariant symbol associated with the mixed state.\\

Let us briefly recollect the notion of coherent states \cite[see ][]{kirwin2005coherent, Berceanu_2000}. For this part of the article, we shall use the Dirac Bra-ket notation wherever it seems appropriate and convenient. Consider a compact connected K\"ahler manifold $M$ with quantum line bundle $L$ that is very ample. Let $\{\theta_j\}_{j=1}^d$ be an orthonormal basis for $H^0(M,L)$. The reproducing kernel known as the generalized Bergman kernel is the section $K $ of $\overline{L}\boxtimes L \to M \times M$ (where $\overline{L} \cong L^*$ using the hermitian structure of the line bundle $L$) given by 
$$K(x,y):= \sum_{j=1}^d \overline{\theta_j(x)} \otimes \theta_j(y) \text{ \quad for } x,y\in M.$$
For each $x \in M$, we define $\Phi_x \in \overline{L_x}\otimes H^0(M,L)$ by $$ \Phi_x := K(x,.) = \sum_{j=1}^d \overline{\theta_j(x)} \otimes \theta_j.$$
Using appropriate trivialization and identifying $1 \otimes \theta_j$ with $\theta_j$, we can think of $\Phi_x$ as a holomorphic section of $L\to M$. Using the reproducing property of the Bergman kernel, we have $\bra{\Phi_x}\ket{s} = s(x)$ for all $x\in M$ and $s\in H^0(M,L)$.
Then the \textit{coherent state} localized at $x\in M$, denoted by $\ket{x}$, is given by 
\[
\ket{x} :=  
    \frac{\Phi_x}{\|\Phi_x\|}.
\]
Note that $\Phi_x \ne 0$ for any $x\in M$ due to the very ample condition.
%We note that for any section $s\in H^0(M,L)$, we have $\bra{x}\ket{s} = \frac{s(x)}{\|\Phi_x\|}$. 
The\textit{ covariant symbol} $\hat{\sigma}$ associated to the operator $\sigma$ of $H^0(M,L)$ is given by $$\hat{\sigma}(x):= \expval{\sigma}{x} \text{ for all } x\in M.$$

\begin{theorem} \label{vanishing_of_cov_symb}
    Let $\sigma$ be a mixed state which is an orthogonal projection such that Ran$(\sigma)^{\perp}$ has a basis consisting of coherent states and $\hat{\sigma}$ be the covariant symbol associated with $\sigma$. Let $\mathcal{V}(\sigma) = \{x \in M_1 \times M_2: \hat{\sigma}(x)=0\} $, then $\rho_{\mathcal{V}(\sigma)}^{\ker}= \sigma$.
\end{theorem}
\begin{proof}
    There exists an orthonormal set $\{s_1,...,s_k\}$ in $\HAB$ such that $\sigma  = \frac{1}{k}\sum_{j=1}^k\ket{s_j}\bra{s_j}$. Denote the range span$\{s_1,...,s_k\} $ of $\sigma$ by Ran$(\sigma)$. Then
    \begin{align*}
        \mathcal{V}(\sigma) & = \{x \in M_1 \times M_2: \expval{\sum_{j=1}^k  \ket{s_j}\bra{s_j}}{x}=0 \} \\
        & = \{x \in M_1 \times M_2: \sum_{j=1}^k |\bra{x}\ket{s_j}|^2=0 \}\\
        & = \{x \in M_1 \times M_2:  \bra{x}\ket{s_j}=0 \text{ for all } j \}\\
        & = \{x \in M_1 \times M_2: \text{ the coherent state }\ket{x} \text{ is in Ran}(\sigma)^{\perp}\}
    \end{align*}
    Due to the hypothesis that Ran$(\sigma)^{\perp}$ has a basis consisting of coherent states, the above set $\mathcal{V}(\sigma)$ is non-empty.
    Now we find the kernel of the restriction map to $\mathcal{V}(\sigma)$. We have
    \begin{align*}
        \ker(\lr_{\mathcal{V}(\sigma)}) & = \{s\in \HAB: s(x)=0 \text{ for all } x \in \mathcal{V}(\sigma)\}\\
        & = \{s\in \HAB: \bra{x}\ket{s}=0 \text{ for all coherent states } \ket{x} \text{ in Ran}(\sigma)^{\perp}\}
    \end{align*}
    Let $s\in $ Ran$(\sigma)$, then $\bra{s'}\ket{s}=0$ for all $s' \in $ Ran$(\sigma)^{\perp}$. In particular, $\bra{x}\ket{s}=0$ for all coherent states $\ket{x} \in $ Ran$(\sigma)^{\perp}$, i.e. $s \in \ker(\lr_{\mathcal{V}(\sigma)})$, i.e. Ran$(\sigma) \subset \ker(\lr_{\mathcal{V}(\sigma)})$.\\

    Conversely, suppose $s \in \ker(\lr_{\mathcal{V}(\sigma)})$, i.e. $\bra{x}\ket{s}=0$ for all coherent states $\ket{x} \in $ Ran$(\sigma)^{\perp}$. Because there is a basis for Ran$(\sigma)^{\perp}$ consisting of coherent states, we see that $\bra{s'}\ket{s} $ for all $s' \in $ Ran$(\sigma)^{\perp}$, i.e. $s \in  $ Ran$(\sigma)$.\\

    Therefore, we see that $ \ker(\lr_{\mathcal{V}(\sigma)}) = $ Ran$(\sigma)$. Hence $\rho_{\mathcal{V}(\sigma)}^{\ker}$ and $\sigma$, both being the orthogonal projection onto the same subspace and having unit trace, are equal. %Hence $\rho_{\mathcal{V}(\sigma)}^{\ker}$  and $\rho_{\mathcal{V}(\sigma)}^{\ker^{\perp}}$ are separable if and only if $\sigma$ is separable.
\end{proof}

%%%......................................%%%%%%%%
%................Appendix................%%%%%%
%%%%%......................................%%%%%%%%
\section*{Appendix}
\noindent\textbf{Proof of lemma \eqref{sumDistIntersec}}\\
    It is straightforward to verify that $\left(\bigcap_{j\in \mathcal{I}}W_j\right)\oplus\left(\bigcap_{j\in \mathcal{I}}V_j\right) \subset \bigcap_{j\in \mathcal{I}}(W_j\oplus V_j) $. For the converse, let $u\in \bigcap_{j\in \mathcal{I}}(W_j\oplus V_j)$. Then for each $j\in \mathcal{I}$, there exists unique $w_j \in W_j$ and $v_j \in V_j$ such that $u= w_j \oplus v_j$. Due to the hypothesis $V \cap W = \{0\}$, we have a well-defined subspace $V\oplus W$ of $X$ and we have $\bigcap_{j\in \mathcal{I}}(W_j\oplus V_j) \subset W \oplus V$. So, there exists unique $w\in W$ and $v\in V$ such that $u = w\oplus v$. For each $j$, since $w_j \in W_j \subset W$, $v_j\in V_j \subset V$ and $u=w_j \oplus v_j$, so by the uniqueness of the decomposition, we must have $w=w_j$ and $v=v_j$, i.e. $w\in W_j$ and $v\in V_j$ for all $j\in \mathcal{I}$, i.e. $w\in \bigcap_{j\in \mathcal{I}}W_j$ and $v\in \bigcap_{j\in \mathcal{I}}V_j$. Therefore, $u= w \oplus v \in \left(\bigcap_{j\in \mathcal{I}}W_j\right)\oplus\left(\bigcap_{j\in \mathcal{I}}V_j\right) $.
\qed\\

\section*{Acknowledgement}
We extend our gratitude to T. Barron for valuable discussions and suggestions, as well as for reading the pre-print to improve the exposition of the article.

%Table of Contents.
%% ***   Set the bibliography style.   ***
 % (change according to your preference)
%%% ***   Set the bibliography file.   ***
\printbibliography

%\bibitem{bk:22}T. Barron, A. Kazachek.  
%\newblock {\em Entanglement of mixed states in K\"ahler quantization., }
%\newblock  Submitted, under review. 2022. 

\end{document}